\documentclass{amsart}

\usepackage[utf8]{inputenc}
\usepackage[english]{babel}
\usepackage{enumitem}

\usepackage[T1]{fontenc}

\usepackage{mathtools}  
\usepackage[dvipsnames]{xcolor}
\usepackage{stmaryrd}
\usepackage{MnSymbol}

\theoremstyle{plain}
\newtheorem{thm}{Theorem}[section]
\newtheorem{theorem}[thm]{Theorem}

\newtheorem{proposition}[thm]{Proposition}

\newtheorem{lemma}[thm]{Lemma}

\newtheorem{corollary}[thm]{Corollary}

\theoremstyle{definition}
\newtheorem{definition}[thm]{Definition}

\newtheorem{example}[thm]{Example}

\newtheorem{remark}[thm]{Remark}

\definecolor{linkred}{RGB}{255,1,1}
\definecolor{citegreen}{RGB}{1,190,1}
\usepackage[linkcolor=linkred
           ,citecolor=citegreen
           ,colorlinks
           ,bookmarksopen=True,  
           ]{hyperref}

\setlist[enumerate,1]{label=\textup{(\arabic*)},
  ref=\textup{(}\arabic*\textup{)}}

\setlist[enumerate,2]{label=\textup{(}\roman*\textup{)},
  ref=\textup{(}\roman*\textup{)}}
\newcommand{\defit}[1]{\textbf{#1}}

\newcommand{\bN}{\mathbb N}

\newcommand{\bZ}{\mathbb Z}

\newcommand{\diff}{D}
\newcommand{\hmnser}[2]{{#1}(\!({#2})\!)}

\DeclarePairedDelimiter{\card}{\lvert}{\rvert}
\DeclarePairedDelimiter{\length}{\lvert}{\rvert}

\usepackage{microtype}

\begin{document}

\title{Monoid algebras and graph products\textsuperscript{*}}
\thanks{\textsuperscript{*} Dedicated to Dragan Maru\v{s}i\v{c} on the occasion of his 70\textsuperscript{th} birthday.}

\author{Wilfried Imrich}
\address{Montanuniversit\"{a}t Leoben, Leoben, Austria}
\email{imrich@unileoben.ac.at}

\author{Igor Klep\textsuperscript{\textdagger}}
\address{Faculty of Mathematics and Physics (FMF),
University of Ljubljana and
Fakulteta za matematiko, naravoslovje
in informacijske tehnologije (FAMNIT), University of Primorska, Koper
and Institute of Mathematics, Physics and Mechanics (IMFM),
Ljubljana, Slovenia}
\email{igor.klep@fmf.uni-lj.si}
\thanks{\textsuperscript{\textdagger} Supported by the Slovenian Research Agency program P1-0222 and grants J1-50002, J1-2453, N1-0217, J1-3004}

\author{Daniel Smertnig\textsuperscript{\ddag}}
\address{Faculty of Mathematics and Physics (FMF),
University of Ljubljana
and Institute of Mathematics, Physics and Mechanics (IMFM),
Ljubljana, Slovenia}
\email{daniel.smertnig@fmf.uni-lj.si}
\thanks{\textsuperscript{\ddag} Supported by the Slovenian Research and Innovation Agency (ARIS): Grant P1-0288 and the Austrian Science Fund (FWF): P~36742}


\begin{abstract}
In this note, we extend results about unique $n^{\textrm{th}}$ roots and cancellation of finite disconnected graphs with respect to the Cartesian, the strong and the direct product, to the rooted hierarchical products, and to a modified lexicographic product. We show that these results also hold for graphs with countably many finite connected components, as long as every connected component appears only finitely often (up to isomorphism).
The proofs are via monoid algebras and generalized power series rings.
\end{abstract}

\maketitle

\section{Introduction}
It is well known that prime factorization of disconnected graphs with respect to the Cartesian product is in general not unique. Nonetheless,   Fern\'andez, Leighton, and
L\'opez-Presa  \cite{fele07} showed the uniqueness of $n^{\textrm{th}}$ roots for
disconnected graphs with respect to the Cartesian product, and  Imrich, Klav\v{z}ar and Rall \cite{imkl07} proved the cancellation property. They also extended the result to the strong and the direct product, but the question whether these properties extend to graphs consisting of countably many finite connected components remained open.

We provide an affirmative answer for these properties in the context of the rooted hierarchical product, the rooted generalized hierarchical product, and a modified lexicographic product.
We first show that the unique root and cancellation properties hold for the rooted generalized hierarchical product of countable graphs consisting of finite connected components of finite multiplicities,  and then extend the result to the other products.

The proofs involve monoid algebras, which are associated with the products, and generalized power series rings.
Along the way, we state sufficient conditions for a semiring of graphs under a given graph product to arise as a monoid semiring, and hence for such a graph semiring to embed into a monoid algebra.

All products considered here are defined on the Cartesian product of the vertex sets of the factors.
They are associative, and admit a neutral element. With the exception of the lexicographic product, they are right and left distributive on their domains with respect to the disjoint union. The lexicographic product is only right distributive.  Apart from the direct and the lexicographic product, they are disconnected if and only if at least one factor is disconnected.
We do not allow multiple edges,  and loops only for the direct product, and a variant of the Cartesian product.

We call a  graph \defit{prime} if it is non-trivial, that is, if it has at least two vertices,  and if it cannot be represented as a product of two non-trivial graphs. This differs from the  terminology of factorization theory, where one would speak of irreducible graphs.  Each finite graph clearly has a representation as a product of prime graphs,  and a representation as a product of prime graphs is a \defit{prime factorization}. If the presentation is unique up to isomorphisms and the order of commuting prime factors, we speak of the \defit{unique prime factorization property}.
Usually, only connected graphs have unique prime factorizations.

The disconnected graphs that we consider are countable with finite connected components, all of which are of finite multiplicity.

\section{Rooted hierarchical products and rooted generalized  hierarchical products}\label{sec:graphproducts}

The rooted generalized hierarchical product, introduced in 2009 by  Barri{\`e}re,  Dalf{\'o}, Fiol and Mitjana \cite{bado2009},  gives rise to the most interesting monoid algebra considered here. Let $G$, $H$ be graphs with distinguished, nonempty subsets $U\subseteq V(G)$, $V\subseteq V(H)$,  the \defit{root sets} of $G$ and $H$. Then the \defit{generalized rooted hierarchical product} $G[U]\sqcap H[V]$ has vertex set $V(G\times H)=V(G) \times V(H)$, root set $U\times V$, and a pair of vertices $(g,h)$, $(g',h')$ is joined by an edge if either $g=g'$ and $hh' \in E(H)$, or if $gg' \in E(G)$ and $h=h' \in V$.
The asymmetry of $G$ and $H$ in the definition
accounts for non-commutativity.

This product is closely related to the \defit{Cartesian product} $G\Box H$ of graphs without root sets, where $(g,h)(g',h')$ is an edge if either $g=g'$ and $hh' \in E(H)$, or if $gg' \in E(G)$ and $h=h'$. Clearly, the Cartesian product is commutative.

Furthermore, as the edge sets of $G[V(G)]\sqcap H[V(H)]$ and $G\Box H$ are the same, we can consider the Cartesian product as a special case of the rooted generalized hierarchical product, and call a factor \defit{Cartesian} if its root set equals its vertex set.

The neutral element of the generalized rooted hierarchical product is the trivial rooted one-point graph $K_1$.

Imrich, Kalinowski and Pil\'{s}niak \cite[Theorem 5]{imrich-kalinowski-pilsniak23} show that finite connected graphs with nonempty root sets have the unique prime factorization property with respect to the rooted generalized  hierarchical product.
This extends the unique prime factorization property of connected finite graphs with respect to the Cartesian product, which  was first shown in 1960 by Sabidussi \cite{sabi1960} and independently, in 1963,  by Vizing \cite{vi1963}.

As such the rooted generalized hierarchical product is not right distributive, unless one restricts its domain to graphs where each connected component has a non-empty root set. Clearly, this class of graphs is closed with respect to the rooted generalized hierarchical product and coincides,  on connected graphs, with the rooted hierarchical product. We call it  the \defit{modified rooted generalized hierarchical product.}

A graph whose root set consists of exactly one vertex is called \defit{rooted}. If one restricts the domain of the generalized rooted hierarchical product to rooted graphs one obtains the \defit{rooted hierarchical product}, which was introduced in 2009 by Barri{\`e}re,  Comellas,  Dalf{\'o} and  Fiol \cite{baco2009}. Again, it is not right distributive,  unless  one restricts its domain  to graphs where each connected component is rooted. We call it the  \defit{modified rooted hierarchical product.} On connected graphs  it coincides with the rooted hierarchical product.

This product is strictly non-commutative. In other words, two connected non-trivial rooted graphs $G[g]$ and $H[h]$ commute if and only if there exists a  rooted graph $A[a]$, such that $G[g]$ and $H[h]$ are powers of $A[a]$.

If $M$ is a class of rooted graphs closed with respect to the rooted hierarchical product, then the elements of $M$ form a monoid with respect to the induced operation (under the assumption that the isomorphism classes indeed form a set, to avoid set-theoretic issues).\footnote{To be precise, we consider a monoid of \emph{isomorphism classes} of graphs. However, in the interest of brevity and to follow the usual conventions, we speak of graphs with the tacit understanding that they are to be considered up to isomorphisms of rooted graphs, respectively multiply rooted graphs.}
If $M$ is taken to be the class of rooted connected finite graphs, then the resulting monoid is a free monoid on countably many generators by  \cite[Corollary 2.6]{imrich-kalinowski-pilsniak23}  (with the generators corresponding to \defit{rooted prime graphs}, the prime elements with respect to the particular operation).

For the rooted generalized hierarchical product one considers multiply rooted graphs.
By \cite[Theorem 4.2]{imrich-kalinowski-pilsniak23}
one obtains a unique decomposition in which two factors commute if and only if they are isomorphic or if they are both Cartesian.
More formally, the monoid under discussion is isomorphic to the coproduct of monoids
\[
  [Y] * \langle X \rangle,
\]
where $[Y]$ is a free abelian monoid on the countable set $Y$ (corresponding to the Cartesian rooted prime graphs) and $\langle X \rangle$ is a free monoid on the countable set $X$ (corresponding to the non-Cartesian rooted prime graphs).
The notations $[Y]$, $\langle X \rangle$ are non-standard but convenient in the present paper.

\section{Monoid algebras and generalized power series rings}\label{sec:monoidalgebras}

Let $M$ be a set of graphs that is closed under some associative product $\cdot$ and possesses a multiplicative identity element $1$.
Then $(M,\cdot, 1)$ forms a monoid.
Suppose that $M$ is also closed under disjoint unions $+$ and contains the empty graph $0$.
Then $(M, +, 0)$ forms a commutative monoid, with neutral element 0.
The monoid $(M, +, 0)$ is always cancellative, that is $a + b = a + c$ implies $b=c$, by the uniqueness of the connected components of a graph.

If the graph product under consideration is distributive with respect to the disjoint union, then the two structures combine to give a \defit{semiring} $(M, +, \cdot, 0, 1)$, with disjoint unions serving as addition and the graph product serving as multiplication.
Formally, the structure $(M,+,0)$ is a commutative monoid, the structure $(M,\cdot,1)$ is a monoid, and the following hold for all $a$, $b$,~$c \in M$:
\[
  c(a+b)=ca+cb, \qquad (a+b)c = ac + bc, \qquad 0a=0=a0.
\]
In contrast to a ring, the addition need not have inverses.
Indeed, in our case, addition is the disjoint union of graphs, so only the empty graph has an additive inverse.

More information on semirings can be found in Hebisch and Weinert \cite{hebisch-weinert98}, although the notion of a semiring is slightly more permissive there, in that it does not require a $0$. The semirings under consideration here are semirings with \emph{absorbing zero} in the terminology of Hebisch and Weinert.

\begin{definition}
  A \defit{graph semiring} is a semiring consisting of \textup(isomorphism classes of\textup) graphs, with the disjoint union of graphs as addition.
\end{definition}

In the following, we establish that, under mild conditions, graph semirings embed as sub-semirings into monoid semirings and even into (well-understood) monoid algebras.
We also establish sufficient conditions for these (semi)rings to be strictly ordered.
All of the arguments are straightforward, but since they draw from a setting that we assume most potential readers not to be very familiar with, we carry them out in some detail.

Let $S$ be a semiring and $(M,\cdot,1)$ a multiplicative monoid.
The monoid semiring of $M$, denoted by $S[M]$ consists of formal $S$-linear combinations of elements of $M$, with componentwise addition and the product extended linearly from the one on $M$.
Thus, an element of $S[M]$ is a formal sum
\[
  f=\sum_{m \in M} f_m m, \qquad\text{with $f_m \in S$, only finitely many of which are nonzero,}
\]
and where the coefficients $f_m$ are uniquely determined by the element $f$.
More formally, we can consider $S[M]$ to be the set of functions $f\colon M \to S$.
If $S$ is a ring, then so is $S[M]$.
We will only need to consider the monoid semiring $\bN_0[M]$ and the monoid algebra $\bZ[M]$.

The inclusion map $j \colon M \to \bN_0[M]$ is a multiplicative homomorphism.
The monoid semiring $\bN_0[M]$ is characterized by the following universal property: for every homomorphism of multiplicative monoids $f\colon M \to T$, with $T$ a semiring, there exists a unique semiring homomorphism $\overline{f}\colon \bN_0[M] \to T$ such that $f = \overline{f} \circ j$.

The monoid algebra $\bZ[M]$ satisfies the analogous universal property in the category of rings (that is, when $T$ is a ring).

\begin{proposition} \label{p:iso-monoidsemiring}
  Let $M$ be a graph semiring with the following properties.
  \begin{enumerate}
  \item \label{isosemi:submonoid} The connected graphs in $M$ form a submonoid $M_0$, that is, products of connected graphs are connected and the multiplicative identity is connected.
  \item \label{isosemi:span} Every graph in $M$ has a finite number of connected components.
  \end{enumerate}
  Then
  \[
    \bN_0[M_0] \cong M.
  \]
\end{proposition}

\begin{proof}
  By the universal property of the monoid semiring, there exists a semiring homomorphism $\pi \colon \bN_0[M_0] \to M$ with $\pi(m)=m$ for all $m \in M_0$.
  By condition~\ref{isosemi:span}, $\pi$ is surjective.
  To show injectivity, let
  \[
  f = \sum_{m \in M_0} f_m m \quad\text{and}\quad g = \sum_{m \in M_0} g_m m \in \bN_0[M_0]
  \]
  with $\pi(f)=\pi(g)$.
  Then $\pi(f)$ has precisely $f_m$ connected components isomorphic to $m$, and $\pi(g)$ has $g_m$ such components. Hence $f_m=g_m$.
\end{proof}

\begin{remark}
  \begin{enumerate}
    \item Condition~\ref{isosemi:span} is clearly necessary for surjectivity, as we have no way of representing a graph with infinitely many isomorphic connected components in $\bN_0[M_0]$.
    This can easily be fixed by replacing $\bN_0$ with a larger semiring, involving infinite cardinals and their cardinal arithmetic.
    In the cases of interest to us, we will later exploit $\bN_0[M]\setminus\{0\}$ being multiplicatively cancellative, which does not generalize to infinite cardinals due to the non-cancellative nature of cardinal arithmetic.
    See Example~\ref{exm:infinite-bad}.

    \item Not every graph product preserves connectivity, for instance, the direct product of two connected bipartite graphs is not connected.
  \end{enumerate}
\end{remark}

We now slightly extend our notions to permit infinitely many connected components in our graphs, as long as each connected component has finite multiplicity (Example~\ref{exm:infinite-bad} illustrates why this condition is  natural).
Order relations on monoids and semirings will play a central role in the following arguments.

\begin{definition}
  \mbox{}
  \begin{enumerate}
    \item A monoid $(M,\cdot, 1)$ is \defit{strictly ordered} if there exists a \textup(total\textup) order $\le$ on $M$ such that
    \[
    a < b \qquad\text{implies}\qquad ac < bc \text{ and } ca < cb \qquad\text{for all $a$, $b$,~$c \in M$.}
  \]
  \item A monoid is \defit{strictly well-ordered} if it is strictly ordered by a well-order.
  \item A semiring $S$ is \defit{strictly ordered} if there exists a \textup(total\textup) order on $S$ such that for all $a$, $b$,~$c \in S$, with $a < b$,
  \[
  a + c < b + c, \qquad\text{if $c > 0$, then } ac < bc, \qquad\text{and}\qquad\text{if $c < 0$, then } ac > bc.
  \]
  \end{enumerate}
\end{definition}

Let $M$ be a strictly well-ordered monoid and let $S$ be a semiring.
We define a \defit{generalized power series ring} $\hmnser{S}{M}$ as follows.
The elements $f$ of $\hmnser{S}{M}$ are functions $f \colon M \to S$, which we represent in a power series notation
\[
f = \sum_{m \in M} f(m) m.
\]
The addition is defined pointwise, and the multiplication is defined in the usual way using the Cauchy product
\[
fg(m) = \sum_{\substack{x, y \in M \\ m=xy}} f(x)g(y).
\]
Crucially, the strict well-order on $M$ ensures that the sum on the right sides has only finitely many terms, so that the expression indeed makes sense.
It is then straightforward to verify that $\hmnser{S}{M}$ is a semiring.
If $S$ is a ring, then $\hmnser{S}{M}$ is even a ring.
Again, we will only be interested in the semiring $\hmnser{\bN_0}{M}$ and the ring $\hmnser{\bZ}{M}$.

\begin{remark}
  The definition of $\hmnser{S}{M}$ is a mild variation on various existing definitions.
  For a strictly ordered monoid $M$ and a (semi)ring $R$ one can define the ring of generalized power series $R(\!(M)\!)$ consisting of all series having well-ordered support.
  The case when $M$ is a group (and typically $S$ is a field) is that of Hahn--Mal'cev--Neumann series.
  For $M$ a commutative monoid and $S$ a commutative ring, the notion has been studied in particular by Ribonboim in the 1990s in a series of papers, see for instance \cite{ribenboim92}.
  More general noncommutative and skew versions have recently also been intensely studied, we mention \cite{mazurek-paykan17,tuganbaev20} as entry points into the extensive literature.
  In the context of semirings, the semiring of formal power series $S\llangle X \rrangle$, which is $\hmnser{S}{M}$ with $M=\langle X \rangle$ a finitely generated free monoid, appears in the study of noncommutative rational series and weighted automata \cite{berstel-reutenauer11}.
\end{remark}

\begin{proposition} \label{p:iso-series}
  Let $M$ be a graph semiring with the following properties.
  \begin{enumerate}
  \item \label{isoseries:submonoid} The connected graphs in $M$ form a submonoid $M_0$.
  \item \label{isoseries:fin} No graph in $M$ has infinitely many isomorphic connected components.
  \item \label{isoseries:all} The set $M$ contains all \textup(infinite\textup) disjoint unions of elements of $M_0$, subject to the restriction \textup{\ref{isoseries:fin}}.
  \item \label{isoseries:order} The monoid $M_0$ is a strictly well-ordered monoid.
  \end{enumerate}
  Then
  \[
    \hmnser{\bN_0}{M_0} \cong M.
  \]
\end{proposition}

\begin{proof}
  The conditions~\ref{isoseries:submonoid} and \ref{isoseries:order} ensure that $\hmnser{\bN_0}{M_0}$ is well-defined.
  The map
  \[
  f\colon \hmnser{\bN_0}{M_0} \to M,\quad \sum_{m \in M_0} f_m m \mapsto \sum_{m \in M_0} f_m m,
  \]
  is well-defined by \ref{isoseries:all} (recall that the sum on the right is the disjoint union of graphs, whereas the sum on the left just stems from the power series notation used for $\hmnser{\bN_0}{M_0}$).
  It is routine to check that $f$ is a semiring homomorphism.
  The uniqueness of the connected components again implies that $f$ is injective.
  Condition~\ref{isoseries:fin} ensures that $f$ is surjective.
\end{proof}

\begin{proposition} \label{p:semiring-order}\mbox{}
  Let $S$ be a strictly ordered semiring and let $M$ be a strictly ordered monoid.
  \begin{enumerate}
    \item $S[M]$ is a strictly ordered semiring.
    \item If the order on $M$ is a well-order, then $\hmnser{S}{M}$ is a strictly ordered semiring.
  \end{enumerate}
\end{proposition}

\begin{proof}
  We first deal with $\hmnser{S}{M}$, as $S[M]$ is similar but easier.
  Let $f \ne g \in \hmnser{S}{M}$.
  Let $m \in M$ be the smallest element for which $f(m) \ne g(m)$ (the minimum exists because $M$ is well-ordered).
  We define $f < g$ if  $f(m) < g(m)$.
  The relation $\le$ is reflexive and anti-symmetric.

  For transitivity, let $f < g$ and $g < h$. Let $m$,~$m' \in M$ be the smallest elements for which $f(m) \ne g(m)$, respectively, $g(m') \ne h(m')$.
  If $m=m'$, then transitivity of the order on $S$ implies $f <h$.
  If $m < m'$, then $f(m) < g(m) = h(m)$ and $f(m'')=g(m'')=h(m'')$ for all $m'' < m$.
  If $m > m'$, then $f(m')=g(m')<h(m')$ and $f(m'')=g(m'')=h(m'')$ for all $m'' < m'$.
  In any case $f < h$.

  If $f < g$, then $f+h < g+h$ follows from the coefficient semiring $S$ having the same property.
  Note that $h > 0$ if and only if $h \ne 0$ and the coefficient of the smallest element of the support of $h$ is positive in $S$.
  Let $f < g$ and $h> 0$. We have to show $fh < gh$.
  If $x < x'$ and $y < y'$ in $M$, then $xy < xy' < x'y'$.
  Let $m \in M$ be the smallest element for which $f(m) \ne g(m)$, and let $m_0 \in M$ be the smallest element for which $h(m_0) \ne 0$.
  Then
  \[
  fh(mm_0) = f(m)h(m_0) + \sum_{\substack{x,y \in M \\ mm_0=xy,  (x,y) \ne (m,m_0)}} f(x)h(y).
  \]
  For any nonzero summand in the right sum, we have $y > m_0$, and hence $x < m$.
  But then $f(x)=g(x)$, so that
  \[
  fh(mm_0) =  f(m)h(m_0) + \sum_{\substack{x,y \in M \\ mm_0=xy,  (x,y) \ne (m,m_0)}} g(x)h(y).
  \]
  Since also
  \[
  gh(mm_0) =  g(m)h(m_0) + \sum_{\substack{x,y \in M \\ mm_0=xy,  (x,y) \ne (m,m_0)}} g(x)h(y),
  \]
  the strict order on $S$ gives $gh(mm_0) > fh(mm_0)$ (we used $h(m_0) > 0$ and $f(m) < g(m)$).
  If $m' < mm_0$, then a similar argument shows $gh(m')=fh(m')$, so that altogether $fh < gh$.

  If $h < 0$, we find $fh > gh$ in the same way using $h(m_0) < 0$.
  Hence $\hmnser{S}{M}$ is a strictly ordered semiring.

  In the same way, it follows that $S[M]$ is a strictly ordered semiring.
  Because the support of every element is finite, each such non-empty set has a minimum, even when $M$ is not well-ordered.
\end{proof}

\begin{proposition} \label{p:semiring-unique}
  Let $S$ be a strictly ordered semiring, and let $a$,~$b \in S$, and $c \in S \setminus \{0\}$.
  \begin{enumerate}
    \item If $ac=bc$, then $a=b$, and symmetrically, if $ca=cb$, then $a=b$.
    \item If $a$,~$b \ge 0$ and $a^n = b^n$ for some $n \ge 1$, then $a = b$.
  \end{enumerate}
\end{proposition}

\begin{proof}
  We first show cancellativity.
  Suppose that $a \ne b \in S$.
  Without restriction $a < b$.
  Let $0 \ne c \in S$.
  If $c > 0$, then $ac < bc$, and hence $ac \ne bc$.
  If $c < 0$, then $ac > bc$, and again $ac \ne bc$.

  If $a^n=b^n$ for some $n \ge 1$ and $a=0$, then also $b=0$, by the same argument as before.
  Let $a$,~$b \in S$ with $a$,~$b > 0$.
  Suppose $a \ne b$.
  Without restriction $a < b$.
  We show $a^n < b^n$ for all $n \ge 2$ by induction on $n$.
  Indeed, if $a^{n-1} < b^{n-1}$ then $a^n=a^{n-1}a < b^{n-1}a$ and $b^{n-1}a < b^{n-1}b=b^n$.
\end{proof}

To show the cancellation, respectively, unique root property for a graph product, it is therefore sufficient to show that the corresponding graph semiring is strictly ordered.

The free abelian monoid $[Y]$ and the free monoid $\langle X \rangle$ are both strictly well-ordered.
On $[Y]$, we use the lexicographical order.
Explicitly, after choosing an arbitrary order on $Y$, elements $a \in [Y]$ have a unique representation $a=y_1^{e_1} \cdots y_r^{e_r}$ with $y_1 < \cdots < y_r$ and $e_i > 0$.
We compare two elements by comparing the exponent of the smallest factor (and set $1 \le a$ for all $a$).

In the case of the free monoid $\langle X \rangle$, we use the \emph{shortlex} order.
We choose an arbitrary well-order on the alphabet $X$.
Elements of $\langle X \rangle$ are words in $X$, and we set $a < b$ if either $\length{a} < \length{b}$, or $\card{a} = \card{b}$ and the left-most letter in which $a$ and $b$ differ is smaller in $a$ than in $b$.
Note that the lexicographic order itself (without comparing lengths first) does \emph{not} give a well-order, as in that case, if $x$, $y$ are letters with $x < y$, then
\[
  y > xy > x^2y > \dots,
\]
is an infinite descending chain.

Thus both $[Y]$ and $\langle X \rangle$ are strictly well-ordered monoids.
Before we can show that this well-ordering ascends to the coproduct, we need another lemma.

\begin{lemma} \label{l:wom}
  Let $A$ be a strictly well-ordered monoid.
  \begin{enumerate}
    \item\label{wom:min} For all $a \in A \setminus \{1\}$ one has $a > 1$ and therefore $\min(A)=1$.
    \item\label{wom:ff} Every $a \in A$ has only finitely many representations $a=a_1 \cdots a_n$ with $n \ge 0$ and $a_1$, \dots,~$a_n \in A \setminus \{1\}$.
  \end{enumerate}
\end{lemma}

\begin{proof}
  \ref{wom:min} Let $a \in A \setminus \{1\}$ and suppose $a \not> 1$.
   Then $a < 1$, and $a^{n-1} > a^n$ for all $n \ge 1$.
   Thus $\{\, a^n : n \ge 0 \,\}$ has no minimum, in contradiction to the well-ordering of $A$.

  \ref{wom:ff}
  Recall that a \defit{quasi-order} is a reflexive and transitive relation.
  A \defit{well-quasi-order} is a quasi-order $\preccurlyeq$ such that for every infinite family $(a_i)_{i \in \bN_0}$ there exist $i < j$ with $a_i \preccurlyeq a_j$.

  We consider the free monoid $M = \langle A\setminus\{1\} \rangle$ of all words in $A\setminus \{1\}$.
  To distinguish the products in $M$ and in $A$, we use the symbol $*$ for the product on $M$.
  We have to show that the homomorphism $\pi\colon M \to A$, $a_1 * \cdots * a_m \mapsto a_1 \cdots a_m$ has finite fibers.

  On $M$ we can define the subword quasi-order as follows:
  If $w=a_1 *\cdots *a_m$ and $w'=a_1'* \cdots *a_n'$ then $w \preccurlyeq w'$ if and only if there exists a strictly increasing map $f \colon \{1,\dots,m\} \to \{1,\dots,n\}$ such that $a_i \le a_{f(i)}'$ for all $1 \le i \le m$ (here $\le$ is the well-order on $A$).
  By Higman's lemma (see e.g. \cite[Theorem 4.3]{higman52} or \cite[Theorem 1.6]{milner85}), the quasi-order $\preccurlyeq$ is a well-quasi-order.

  Now suppose that there exists $a \in A$ such that $\pi^{-1}(\{a\})$ is infinite.
  By the well-quasi-ordering on $M$, there exist $w \ne w'$ in $\pi^{-1}(\{a\})$ with $w \preccurlyeq w'$.
  Say $w=a_1 *\cdots *a_m$ and $w'=a_1'* \cdots *a_n'$ with $a_i$,~$a_j' \in A \setminus \{1\}$.
  Let $f\colon \{1,\dots, m\} \to \{1,\dots,n\}$ be strictly increasing with $a_i \le a_{f(i)}'$.
  Then $w \ne w'$ implies $n > m$ or $a_{f(i)}' > a_{i}$ for some $i$.
  But then $a = a_1' \cdots a_n' > a_1 \cdots a_m=a$ in $A$, a contradiction.
\end{proof}

\begin{remark}
  The second claim of Lemma~\ref{l:wom} implies in particular that $A$ is a \emph{finite factorization monoid} (every element has only finitely many distinct factorizations into irreducibles).
  The application of Higman's lemma here is reminiscent of the one of Cossu and Tringali in \cite[Theorem 4.11]{cossu-tringali23}, where they deduce that for certain monoids the bounded factorization property already implies the finite factorization property.
  Unlike in their proof, we do not need to assume a priori that elements in $A$ only have factorizations of bounded lengths, because the well-ordering on $A$ together with Higman's lemma implies that.
\end{remark}

\begin{lemma} \label{l:order-coproduct}
  If $A$ and $B$ are strictly well-ordered monoids, then the coproduct $A * B$ is also strictly well-ordered.
\end{lemma}

\begin{proof}
  Since $A$ and $B$ are in particular strictly (totally) ordered, the coproduct $A * B$ can also be strictly ordered (see \cite{johnson68} or \cite[Theorem 16]{bergman90} for this non-trivial fact).
  Moreover, taking a lexicographical order on the Cartesian product $A \times B$, we can assume that the order on $A*B$ is such that the canonical monoid homomorphism $\pi \colon A * B \to A \times B$ is order-preserving, meaning $f \le g$ implies $\pi(f) \le \pi(g)$ (again \cite{johnson68} or \cite[p.322, \S5, paragraph after (20)]{bergman90}).

  We show that this order on $A*B$ is a well-order.
  Every element $f \in A*B$ can be represented as $f=a_1b_1 a_2b_2 \cdots a_m b_m$ with $a_1 \in A$, $b_m \in B$ and all other $a_i \in A \setminus \{1_A\}$ and $b_i \in B \setminus \{1_B\}$ (the first and last letter are unrestricted to permit representations that start with a factor in $B$ or end with a factor in $A$).
  Then $\pi(f) = (a_1\cdots a_m, b_1\cdots b_m)$.
  If $g=a_1'b_1' a_2'b_2' \cdots a_n' b_n'\in A * B$ is such that $\pi(f)=\pi(g)$, then $a_1 \cdots a_m = a_1'\cdots a_n'$ and $b_1\cdots b_m=b_1'\cdots b_n'$.
  From \ref{wom:ff} of Lemma~\ref{l:wom} we conclude that $\pi$ has finite fibers (in case $a_1=1_A$ we apply the lemma with $a_2\cdots a_m$, and analogously in the cases where some of $a_1'$, $b_m$, $b_n'$ are $1$).

  Now, if $\emptyset \ne  X \subseteq A*B$, then the set $\pi(X)$ has a smallest element $m_0$ by the well-ordering on $A\times B$.
  Since $\pi^{-1}(\{m_0\})$ is finite, there is a smallest $m \in \pi^{-1}(\{m_0\})$.
  We claim $m = \min(X)$.
  Let $f \in X$.
  If $f \not > m$, then $f \le m$.
  Then $\pi(f) \le \pi(m)=m_0$ by the order-preserving nature of $\pi$.
  By minimality of $m_0$, therefore $\pi(f)=\pi(m)$.
  The minimality of $m$ in the fiber shows $f=m$.
\end{proof}

For later use, we summarize the properties in the case most pertinent to us.

\begin{theorem} \label{t:free-semigroup-rings}
  Let $X$, $Y$ be \textup(possibly empty\textup) sets.
  Then the monoid $M\coloneqq [Y] * \langle X \rangle$ is strictly well-ordered, and the semirings
  \[
  \bN_0[M] \quad\text{and}\quad \hmnser{\bN_0}{M}.
  \]
  are strictly ordered.
  In particular, the products are cancellative, and $f^n =g^n$ for some $n \ge 1$ implies $f=g$.
\end{theorem}

\begin{proof}
  The monoid $M$ is well-ordered by Lemma~\ref{l:order-coproduct}.
  Hence Propositions~\ref{p:semiring-order} and \ref{p:semiring-unique} imply the remaining claims (note $f$, $g > 0$ for all nonzero elements of these two semirings).
\end{proof}

\subsection{From semirings to rings}

Theorem~\ref{t:free-semigroup-rings} will be enough for our later applications.
However, since rings are much better-studied objects in algebra as compared to semirings, it is reasonable at this point to take a final small step and to pass from semirings to rings.

To do so, we need to add negative elements, which will work because a graph semiring $S$ is always cancellative with respect to addition (the disjoint union of graphs).

Let $S$ be a semiring.
A ring $R$ is a \defit{ring of differences} for $S$ if there is a semiring homomorphism $j\colon S \to R$ such that every element of $S$ has an additive inverse in $R$, and such that $R$ is universal with respect to that property.
That is, if $T$ is any ring for which there is a semiring homomorphism $f\colon S \to T$, then there exists a ring homomorphism $\overline{f}\colon R \to T$ such that $f=\overline f \circ j$.
Rings of differences are uniquely determined up to unique isomorphism by this universal property.

\begin{lemma}
  Every graph semiring $S$ can be embedded into a ring of differences $D=D(S)$.
\end{lemma}

\begin{proof}
  This is a consequence of the additive semigroup being cancellative \cite[Theorem II.5.11]{hebisch-weinert98}.
  For the reader's convenience, we sketch the (familiar) construction.
  Let $S$ be a graph semiring, so that $(S,+,0)$ is a cancellative monoid.
  On pairs $(m,n) \in S \times S$ one defines an equivalence relation
  \[
    (m,n) \simeq (m',n') \quad\Leftrightarrow\quad m + n' = m' + n.
  \]
  This is a congruence relation with respect to the additive structure, and $D \coloneqq (S \times S)/\!\simeq$ is an additive group.
  There is an embedding of monoids $j\colon (S,+) \rightarrow (D,+)$, $m \mapsto [(m,0)]_{\simeq}$ and, identifying $S$ with its image under this map, every element $x \in D$ has the form $x = m - n$ with $m$,~$n \in S$.
  Defining on $D$ a multiplication by
  \[
    (m - n)(m' - n') \coloneqq (mm' + nn') - (mn' + nm'),
  \]
  turns $D$ into a ring with sub-semiring $S$.

  If $T$ is a ring and $f\colon S \to T$ is a semiring homomorphism, then $\overline f \colon D \to T$ defined by $\overline f([(m,n)]_{\simeq}) = m - n$ is a well-defined semiring homomorphism satisfying $f = \overline{f} \circ j$.
\end{proof}

\begin{corollary} \label{c:iso-monoidalgebra}
  Let $M$ be a graph semiring with ring of differences $\diff(M)$, and with the following properties.
  \begin{enumerate}
  \item \label{isomonoid:submonoid} The connected graphs in $M$ form a submonoid $M_0$.
  \item \label{isomonoid:span} Every graph in $M$ has a finite number of connected components.
  \end{enumerate}
  Then
  \[
    \bZ[M_0] \cong \diff(M).
  \]
\end{corollary}

\begin{proof}
  We know $\bN_0[M_0] \cong M$ from Proposition~\ref{p:iso-monoidsemiring}.
  The monoid algebra $\bZ[M_0]$ is easily seen to be the ring of differences of $\bN_0[M_0]$, and hence the semiring isomorphism $\bN_0[M_0] \cong M$ extends to a ring isomorphism $\bZ[M_0] \cong D(M)$.
\end{proof}

\begin{corollary} \label{c:iso-ringseries}
  Let $M$ be a graph semiring with ring of differences $\diff(M)$, and with the following properties.
  \begin{enumerate}
  \item \label{isoringseries:submonoid} The connected graphs in $M$ form a submonoid $M_0$.
  \item \label{isoringseries:fin} No graph in $M$ has infinitely many isomorphic connected components.
  \item \label{isoringseries:all} The set $M$ contains all \textup(infinite\textup) disjoint unions of elements of $M_0$, subject to the restriction \textup{\ref{isoringseries:fin}}.
  \item \label{isoringseries:order} The monoid $M_0$ is a strictly well-ordered monoid.
  \end{enumerate}
  Then
  \[
    \hmnser{\bZ}{M_0} \cong D(M).
  \]
\end{corollary}

\begin{proof}
  We know $\hmnser{\bN_0}{M_0} \cong M$ from Proposition~\ref{p:iso-series}.
  Since $\hmnser{\bZ}{M_0}$ is the ring of differences of $\hmnser{\bN_0}{M_0}$, the claim follows.
\end{proof}

Consider the \emph{modified rooted hierarchical product} of finite graphs.
The submonoid of connected graphs is a free monoid on countably many generators.
Hence the associated monoid algebra is simply the free algebra $\bZ\langle X \rangle$ in a countable set $X$, where addition in the algebra corresponds to the disjoint union of graphs.
If we allow infinite graphs where each connected component is finite, and every component appears at most finitely often up to isomorphism, we obtain the ring of formal noncommutative power series in countably many indeterminates $\bZ\llangle X \rrangle$.

For the \emph{generalized rooted hierarchical product}
\[
  \bZ\big[ [Y] * \langle X\rangle \big] \cong \bZ[Y] *_\bZ \bZ\langle X\rangle,
\]
is a coproduct of a polynomial ring in countably many indeterminates and a free algebra in countably many indeterminates.
In the terminology of Cohn \cite{cohn06}, this is the \defit{free $\bZ[Y]$-ring$_{\bZ}$}
\[
  \bZ[Y]_\bZ \langle X \rangle.
\]
on the set $X$.
Allowing arbitrary sets $X$ and $Y$ (or at least permitting $Y=\emptyset$) we recognize the rings of the modified rooted hierarchical product as a special case.

Before we summarize the pertinent information, let us recall that in ring theory a domain is  a non-zero ring in which $ab = 0$ implies $a = 0$ or $b = 0$.

\begin{theorem} \label{thm:main}
  Let $X$, $Y$ be \textup(possibly empty\textup) sets.
  Then the monoid $M\coloneqq [Y] * \langle X \rangle$ is strictly well-ordered, and the rings
  \[
  \bZ[M]\cong \bZ[Y] *_\bZ \bZ\langle X \rangle \quad\text{and}\quad \hmnser{\bZ}{M}.
  \]
  are strictly ordered.
  In particular, they are domains, have the cancellation property, and $f^n=g^n$ for some $n \ge 1$ implies $f=\pm g$.
\end{theorem}

\begin{proof}
  Almost everything follows as in Theorem~\ref{t:free-semigroup-rings} (noting that $\bZ$ is strictly ordered under the usual order).
  It only remains to check that $f^n = g^n$ for some $n \ge 1$ implies $f=\pm g$ even if one (or both) of $f$,~$g$ are negative.
  If $f$ and $g$ are both negative, we can replace them by $-f$ and $-g$ and get $-f=-g$ from the positive case, hence $f=g$.
  If $f < 0 < g$ and $f^n = g^n$, then $n$ must be even, and hence $(-f)^n=g^n$.
  Thus $f=-g$.

  For the cancellation property, we could invoke Proposition~\ref{p:semiring-unique}, but it is also well-known that a nonzero ring is a domain if and only if its nonzero elements are cancellative.
\end{proof}

If $X=0$, then $\bZ[M]=\bZ[Y]$ is a commutative polynomial ring in possibly infinitely many variables (this case arises, for instance, for the Cartesian product of finite graphs).
In this case $\hmnser{\bZ}{Y}=\bZ\llbracket Y \rrbracket$ is a commutative formal power series ring.

If $Y$ is infinite, the ring $\bZ\llbracket Y \rrbracket$ is \emph{not} the completion
\[
  \widehat{\bZ[Y]} \cong \varprojlim \bZ[Y] / I^n.
\]
with respect to the ideal $I$ generated by $Y$.
Indeed, in the quotient $\bZ[Y]/I^n$ every element has only finite support, even though $Y$ is infinite.
Therefore $\widehat{\bZ[Y]}$ is the subring of $\bZ\llbracket Y \rrbracket$ consisting of those series whose support contains only finitely many elements for every given total degree $n$.

\begin{remark}
  \begin{enumerate}
    \item
  While rings are better understood than semirings, we do lose some information about factorizations of graphs in passing to them.
  After all, in a graph, there is no concept of a connected component with negative multiplicity.
  However, allowing negative coefficients significantly alters the factorization properties.

  For a simple example, consider the semiring $M$ of finite graphs under the Cartesian product.
  Connected graphs factor uniquely into prime graphs, so that $M_0=[Y]$ with $Y$ the set of prime graphs.
  Thus $M \cong \bN_0[Y]$ and $D(M) \cong \bZ[Y]$, a polynomial ring in countably many variables.
  The ring $\bZ[Y]$ is a unique factorization domain, but the semiring $\bN_0[Y]$ does not have unique factorizations.
  This observation was used in \cite{fele07} to prove the unique root property for finite Cartesian products, and corresponds to the well-known fact that factorizations of \emph{disconnected} graphs into prime graphs are not unique with respect to the Cartesian product.

  The study of factorizations in semirings has recently gained considerable traction \cite{baeth-gotti20,correamorris-gotti22,chapman-polo23,gotti-polo23,gotti-polo23b,jiang-li-zhu23}.
  In particular, the univariate polynomial semiring $\bN_0[x]$ was recently studied by Cesarz, Chapman, McAdam and Schaeffer \cite{cesarz-chapman-mcadam-schaeffer09} and by Campanini and Facchini \cite{campanini-facchini19}.
  This semiring is not even half-factorial, that is, an element may have factorizations of different lengths.
  Even more, it was shown that, for every rational $r \ge 1$, there exists an $f \in \bN_0[x]$ such that the ratio of the lengths of the longest to the shortest factorization of $f$ is precisely $r$. Moreover, for every $d \ge 0$, there exists some $f\in \bN_0[x]$ and $l \ge 1$ such that $f$ has factorizations of lengths $l$ and $l+d$, but no factorization of length strictly between $l$ and $l+d$. See \cite[Theorem 2.3]{cesarz-chapman-mcadam-schaeffer09} and the comment following it.
  These results immediately imply corresponding ones for Cartesian products of (disconnected) finite graphs.
  Motivated by the case of the Cartesian product, it would be interesting to study the non-unique factorizations of $\bN_0[Y]$ with $Y$ countable (or even finite and $\card{Y} \ge 2$) in more detail.

  \item For graph products in the classical sense of Imrich and Izbicki \cite{imiz1975}, which are defined purely in terms of the adjacency relation (and therefore do not include the various rooted products), it is known that only six products give rise to (possibly non-unital) rings \cite[Theorem 2]{campanelli-friasarmenta-martinezmorales17}.
  \end{enumerate}
\end{remark}

\section{Applications to graph products}\label{sec:applications}

We now apply the main result, Theorem~\ref{thm:main}, to graph products to show
the unique root property and the  cancellation property, as defined below. We begin with
\smallskip

\noindent\textbf{1.  The Generalized Rooted Hierarchical Product.}
By the previous sections it is clear that the ring $R=\hmnser{\bZ}{[Y]*\langle X \rangle}$ arising from the finite connected graphs $G[U]$, $U\neq \emptyset$, that are prime with respect to the generalized rooted hierarchical product satisfies the conditions of Theorem~\ref{thm:main}. Hence, for any two elements $a, b\in R$ of the ring with non-negative coefficients, the equality  $a^n = b^n$ for some  $n\ge 1$, implies that $a=b$.
Furthermore, if $a,b,c \in R$, where $a,b,c$ have only non-negative coefficients, then either of the identities $ca=cb$ or $ac=bc$ implies that $a=b$ (unless $c$ is empty), because $R$ is a domain.
We call the first property the \defit{unique root property}  and the second the \defit{cancellation property}.

\begin{theorem}\label{thm:rghp} Finite and countably infinite graphs, where each connected component is a finite graph with non-empty root set and finite multiplicity,  have the unique root property, and the cancellation property, with respect to the generalized rooted hierarchical product.
\end{theorem}

\begin{proof}
  Let $M$ be the set of (isomorphism classes of) finite and countably infinite graphs, where each connected component is a finite graph with non-empty root set and finite multiplicity.
  Then $M$ is a graph semiring with respect to the modified generalized hierarchical product.
  As discussed after Corollary~\ref{c:iso-ringseries}, the ring of differences $D(M)$ of $M$ is isomorphic to $R=\hmnser{\bZ}{[Y]*\langle X \rangle}$.
  Thus $D(M)$ satisfies the conclusion of Theorem~\ref{thm:main}.

  In particular, suppose $A,B,C$ are finite or countably infinite graphs, where each connected component is a finite graph with non-empty root set.
  Assume that every connected component appears only finitely often up to isomorphism.
  Then the isomorphism classes of  $A$, $B$,~$C$ correspond to elements $a$, $b$,~$c\in R$, all of whose coefficients are non-negative.
  Hence, if $A^n \cong B^n$ for some $n \ge 1$, then $A \cong B$.
  And if $C$ is non-empty and $A \sqcap C \cong B \sqcap C$, then $A \cong B$.
\end{proof}

We now list other products and domains on which they have the unique root and the cancellation property.
\smallskip

\noindent \textbf{2. The Rooted Hierarchical Product.} It has the unique root, and the cancellation property, for finite and countably
infinite graphs, where each connected component is a finite rooted graph and each connected component has finite multiplicity. Note that in this case the free abelian monoid $Y$ in the characterization of $R$ as $\bZ[Y]_\bZ\langle X \rangle$  is empty.
\smallskip

\noindent \textbf{3. The Cartesian Product.} We introduced it as a special case of the
generalized hierarchical product. It thus clearly has
 the unique root and the cancellation property for finite and countably infinite graphs, where each connected component is finite with finite multiplicity. Here $X$ is empty.
\smallskip

Let us remark that the definition of the  Cartesian product extends to the class of graphs with loops, which we denote by $\Gamma_0$. By Boiko, Cuno,  Imrich, Lehner, and van de Woestijne
\cite{bocu2016} the unique prime factorization property holds for all finite connected graphs in $\Gamma_0$ that have at least one unlooped vertex. Hence, in $\Gamma_0$ this product has the unique root and cancellation property for finite and infinite graphs, in which each connected component is finite,  contains at least one unlooped vertex, and has finite multiplicity. Here too, the set $X$ is empty.
\smallskip

One can also extend the Cartesian product to hypergraphs. By Imrich \cite{im1967} connected hypergraphs have the unique prime factorization property, and distributivity with respect to the disjoint union of hypergraphs holds.  Therefore, countably infinite hypergraphs,  where each connected component is finite with finite multiplicity, also have  the unique root and cancellation property.  For a proof in English of the unique prime factorization property for connected hypergraphs we refer to Gringmann \cite{gring2010}.
\smallskip

\noindent
\textbf{4. The Direct Product.}  We  define this product on the class $\Gamma_0$ of graphs in which loops are allowed.
Let $G, H \in \Gamma_0$. Then the edges of the \defit{direct product} $G\times H$ are all pairs of vertices $(g,h)(g',h')$, where $gg'\in E(G)$ and $hh'\in E(H)$. The product is commutative,  and the one vertex graph with a loop is the neutral element. Note that the direct product of loopless graphs is loopless too.

In general, prime factorization with respect to the direct products is not unique, not even for connected graphs. However, it is unique for finite connected non-bipartite graphs, where a graph is \defit{non-bipartite} if it contains at least one cycle of odd length (which can also be a loop).
This follows from a general result of McKenzie \cite{mk1971} about infinite relational structures. For a graph theoretic proof see Imrich \cite{im1998}.

The unique prime factorization for products of finite connected non-bipartite graphs implies that Theorem~\ref{thm:rghp} also holds for the direct product of finite and infinite graphs, if each connected component is finite and non-bipartite, and has finite multiplicity.

For finite graphs in $\Gamma_0$ we have the following theorem of Lov\'asz \cite{lo1971}.

\begin{theorem}\label{thm:lo}
Let $G$,~$H \in \Gamma_0$. If $G^n \cong H^n$ for some $n \ge 1$, where
powers are taken with respect to the direct product, then $G\cong H$.
Furthermore, if $A$, $B$,~$C \in \Gamma_0$ and if there are
homomorphisms from $A$ and $B$ to $C$,  then
$A\times C\cong B\times C$, implies $A\cong B$.
\end{theorem}

Let us see how this compares with the results we can achieve. The indeterminates in our rings are prime graphs and we need unique prime factorization.

This means that Lov\'asz' unique root result is stronger than ours for finite graphs, but it does not cover infinite graphs.

Also, our cancellation result only covers non-bipartite graphs, but for them it is stronger.
\smallskip

\noindent
\textbf{5. The Strong Product.} The strong product is defined in the class of graphs without loops. For  the strong product $G\boxtimes H$ of  graphs $G$,~$H$ one sets $E(G\boxtimes H) = E(G\Box H) \cup E(G\times H)$. Again, the product is commutative, associative and $K_1$ is the neutral element. Furthermore, as shown by D{\"o}rfler and Imrich \cite{doim1970}, each connected finite graph has the unique prime factorization property with respect to the strong product. This implies that all countably infinite graphs with finite connected components, and each connected component appearing only finitely often up to isomorphism, have the unique root and cancellation property. For infinite graphs this is new, but for finite graphs it is also a consequence of Theorem~\ref{thm:lo}.

To see this, just observe that if one first adds loops to all vertices of graphs $G$, $H$, then multiplies the resulting graphs with respect to direct product, and finally removes all loops again, one obtains the strong product $G\boxtimes H$.    The strong product can thus be considered as special case of the direct one. Moreover, since every graph with a loop is non-bipartite, this is not a restriction for the strong product.
This observation already takes care of the unique root property. The cancellation property follows, because there are always homomorphisms from $A$ and $B$
into $C$ if $C$ has at least one loop; one simply  maps $A$ and $B$ into such a vertex and its loop.
\smallskip

\noindent
\textbf{6. The Lexicographic Product.} Given graphs $G$,~$H$, two vertices $(g,h)$, $(g',h')$ are adjacent in the \defit{lexicographic product} $G\circ H$ if either $gg' \in E(G)$, or $g=g'$ and $hh'\in E(H)$. It is non-commutative, associative and $K_1$ is a unit. It is right-distributive with respect to the disjoint union of graphs, but not left-distributive. Hence our methods do not directly apply.

Nonetheless,  by Imrich \cite{im1972},
finite graphs have the unique root and the cancellation property with respect to the lexicographic product.

To extend at least part of this result to infinite graphs we observe that by \cite{im1972} prime factorization of finite connected graphs with respect to the lexicographic product is unique for products of prime connected graphs without a $K_n$, $n> 1$, as a factor. As a matter of curiosity, let us remark that, by \cite{im1972},  for any given finite graph, the number of factors in any prime factorization with respect to the lexicographic product is always the same, despite the fact that prime factorization is in general not unique.

Moreover, by Imrich \cite{im1969},  two finite connected non-trivial graphs commute with respect to the lexicographic product if and only if they are either both complete, edgeless, or powers of one and the same graph.

We can thus consider the free monoid formed by the prime, connected graphs that are not complete. However, due to the failure of left distributivity, we do not obtain a semiring structure.
To obtain left distributivity we introduce the \defit{modified lexicographic product}  by letting it coincide with the lexicographic product for connected graphs and by setting $C\circ (\bigcup_{i\ge 1} A_i)
 = \bigcup_{i\ge 1} (C\circ A_i)$ for disconnected ones with countably many connected components. Clearly the modified lexicographic product also has the  unique root property, as well as the cancellation property.
 \smallskip

\begin{example} \label{exm:infinite-bad}
  When considering infinite graphs, we had to assume that each connected component appears only finitely often up to isomorphism.
This is because on the semiring side, the connected components correspond to indeterminates, and their multiplicity is recorded in the coefficient, which is an element of $\bN_0$.
Of course, we could consider the semiring $S=\bN_0 \cup \{\aleph_0, \aleph_1, \dots, \kappa\}$ of cardinal numbers (up to some bound $\kappa$, to avoid set-theoretic issues) with cardinal arithmetic, and extend our isomorphism result to $\hmnser{S}{M_0} \cong M$.
However, the semiring $S$ is no longer strictly ordered (or even cancellative) and so neither is $M$.

This is not just a defect of our approach.
To see how the cancellation property and the unique root property can fail for associative graph products that are distributive over infinite disjoint unions, let $G$ be some connected finite graph, and let $G'\coloneqq \aleph_0 G$ be a countable disjoint union of copies of $G$.
Then clearly $\aleph_0 G + G = \aleph_0 G + n G$ for all finite $n \ge 0$, so that additive cancellation fails (because $G \ne nG$ if $n \ne 1$).

Suppose now that $G^n$ is connected for all $n$ and that $G$ is not trivial, so that $G^n \ne G^m$ for $m \ne n$ (the powers are taken with respect to the graph product under consideration).
Fix a sequence $(a_n)_{n \ge 2}$ of positive integers, and consider
\[
H = \aleph_0 G + a_2 G^2 + a_3 G^3 + \cdots,
\]
Then
\[
H^2 = \aleph_0 G^2 + \aleph_0 G^3 + \aleph_0 G^4 + \cdots,
\]
independent of the sequence $(a_n)_{n \ge 2}$.
Thus the unique root property fails.
\end{example}

\paragraph{Acknowledgement.} We thank the reviewers for their comments, which helped to improve the final version of the paper.

\bibliographystyle{adamjoucc}
\bibliography{graphproduct}
\end{document}